\newcommand{\q}{\quad}

\def\rr{{\mathbb R}}

\def\rz{{{\rr}^n}}
\def\zz{{\mathbb Z}}

\def\lz{\lambda}

\def\l{\left}
\def\r{\right}

\def\BMO{\mathrm{BMO}}
\def\CBMO{\mathrm{CBMO}}

\newcommand{\supp}{\mathop{\textup{supp}}}
\def\nH{\mathcal {H}}

\newcommand{\f}{\frac}

\documentclass{amsart}
\usepackage{bbm}
\usepackage{amsfonts}

\newtheorem{theorem}{Theorem}[section]
\newtheorem{lemma}[theorem]{Lemma}

\theoremstyle{definition}
\newtheorem{definition}[theorem]{Definition}

\newtheorem{proposition}[theorem]{Proposition}

\theoremstyle{remark}
\newtheorem{remark}[theorem]{Remark}

\numberwithin{equation}{section}

\begin{document}

\title{Endpoint Estimates for N-dimensional Hardy Operators and Their Commutators}

\author{Fayou Zhao}
\address{Department of Mathematics,
Shanghai University, Shanghai, 200444, P.R. China}
\email{zhaofayou2008@yahoo.com.cn}

\author{Zunwei Fu}
\address{Department of Mathematics, Linyi University, Linyi Shandong, 276005, P.R. China}
\email{lyfzw@tom.com}

\author{Shanzhen Lu}
\address{School of Mathematical Sciences, Beijing Normal University, Beijing,
100875, P.R. China} \email{lusz@bnu.edu.cn}
\thanks{The first author was supported in part by NSFC (10871024, 10971130) and
 Shanghai LADP (J50101). The second author was supported in part by NSFC (10871024, 10901076) and
 NSF of Shandong (Q2008A01) and the third author  was supported in part by NSFC (10871024, 10931001).}


\date{Received by the editors ,  and, in revised form, , .}



\begin{abstract}
In this paper, it is proved that the higher dimensional Hardy operator is bounded from Hardy space to Lebesgue space.
The endpoint estimate for the commutator generated by Hardy operator and (central) BMO function is also discussed.
\end{abstract}

\maketitle

%

\section{Introduction}
Let $f$ be in $L^{p}(\mathbb{R}^1)$, and the one-dimensional Hardy operator $H$ be
defined by
$$Hf(x)=\frac{1}{x}\int^x_{0}f(t)\,dt,\,\, x\neq0.$$ A celebrated Hardy integral
inequality [7] can be formulated as
$$\|Hf\|_{L^{p}(\mathbb{R}^1)}\leq \frac{p}{p-1}\|f\|_{L^{p}(\mathbb{R}^1)},$$ where $1<p<\infty$ and the constant $\frac{p}{p-1}$ is
the best possible. In 1995,
Christ and Grafakos in \cite{CG}
gave its $n$-dimensional equivalent version, $\mathcal{H}$,  defined by

$$
  \mathcal{H} (f)(x)=\frac{1}{\upsilon_n|x|^n}\int_{|y|<|x|}f(y)\,
  dy,\,\,\,\,  x\in \rz\backslash\{0\},
$$
where  $\upsilon_n$ is the volume of the unit ball in $\rz$ and they also proved that
$$\|\nH\|_{L^p(\rz)\rightarrow L^p(\rz)}=\frac{p}{p-1},\,\, 1<p<\infty.$$ In what follows, we will work on $\rz$. We pose the first question as follows.

\medskip
\noindent{\textbf{Question 1.}} {\it Is $\nH$ bounded on $L^1(\rz)$}?
\medskip

In fact, it doesn't hold if we take  $f(x)=|x|^{\alpha}\chi_{B(0, R)}(|x|)$, $\alpha>-n$. Here and in what follows, denote by $B(x, R)$ the ball centered at $x$ with radius $R>0$, $|B(x, R)|$ the Lebesgue measure of $B(x, R)$
and $f_{B(x,\, R)}=\frac{1}{|B(x,\, R)|}\int_{B(x,\, R)}f(y)dy.$

The weak $L^1(\rz)$ is defined as the set
of all measurable functions $f$ such that
$$\|f\|_{L^{1,\infty}(\rz)}=\sup_{\lz>0}\lz|\{x\in \rz:|f(x)|>\lz\}|$$
is finite. As we know, the weak $L^1$ spaces are larger than the usual $L^1$ spaces. More recently, we \cite{FLZ1} got

\begin{proposition} \label{prop1}
$\nH$ maps from $L^1(\rz)$ to $L^{1,\infty}(\rz)$ with the sharp bound 1.
\end{proposition}

\begin{definition}\cite{Lu1}\label{d30} A function $a\in L^\infty$ is called a $(1,\infty, 0)$-atom, if it satisfies the following conditions:
(1) $\supp (a)\subset B(x_0, r)$;\,\,
(2) $\|a\|_{\infty}\leq |B(x_0,r)|^{-1}$;\,\,
(3) $\int a(x)dx=0$.
\end{definition}

As a proper subspace of $L^1$, the atomic Hardy space $H^1$ is defined by
$$
H^1=\l\{f\in \mathcal(\rz): f(x)=^{{S}'}\sum_k\lz_ka_k(x),
and \ \sum_k|\lz_k|^p<\infty\r\},\
$$
each $a_k$  is  a $(1, \infty,0)$-atom, and $f$ is a tempered distribution.
Setting $H^1$ norm of $f$ by
$$\|f\|_{H^1}:= \inf \sum_{k}^\infty|\lz_k|.$$
where the infimum is taken over all the decompositions of $f=\sum_k\lz_ka_k$ as above.

Corresponding to Question 1, we naturally ask

\medskip
\noindent{\textbf{Question 2.}} {\it Is $\nH$ bounded on $H^1(\rz)$}?
\medskip

Motivated by the counterexample in  \cite{Gol}, we take \\
$$ b(x)=\left\{\begin{array}{ll}
1-2^n,&\quad \mbox{for $|x|\leq 1,$}
\\ 1,&\quad \mbox{for
$1<|x|\leq 2,$}\\  0,&\quad \mbox{other.}\end{array}\right.$$
Then
$$ \nH(b)(x)=\left\{\begin{array}{ll}
1-2^n,&\quad \mbox{for $|x|\leq 1,$}
\\ 1-2^n|x|^{-n},&\quad \mbox{for
$1<|x|\leq 2,$}\\  0,&\quad \mbox{other.}\end{array}\right.$$
Obviously, $b$ is an atom of $H^1$, however, $\int\nH(b)(x)dx=-n2^n\ln2\neq 0$, therefore
$\nH(b) \not\in H^1(\rz)$.

Since $\nH$ is not bounded on $L^{1}$ or $H^{1}$, a natural question is: does  $\nH$ map $H^1$ to $L^1$? The answer is confirmed. It is just one of our main results in this paper. Our proof is based on the atomic decomposition and certain beautiful and elegant ideas. Finally, we will discuss the endpoint estimate for the commutators generated by Hardy operators and (central) $\BMO$ functions.

Throughout this paper, $A\sim B$ denotes $A$ is equivalent to $B$, that
means there exist two positive constants $c_{1}$ and $c_{2}$ such
that $c_{1}A\leq B\leq c_{2}A$.  For $1<p, p'<\infty$, $p$ and $p'$ are  conjugate indices, that is,
$1/p+1/{p'}=1$. Formally, we will also define $p=1$  as conjugate to $p'=\infty$ and vice versa.

\section{Endpoint estimates for $n$-dimensional Hardy operators}

\begin{theorem}
 $\mathcal {H}$ maps from $H^1(\rz)$ to $L^1(\rz)$.
\end{theorem}

\begin{proof} Assume $a$ is an atom of $H^1$ and satisfies the following conditions:
(i) $\supp (a)\subset B(x_0, r)$,
(ii)$\|a\|_{\infty}\leq |B(x_0,r)|^{-1}$ and
(iii)$\int a(x)dx=0$. We now take $\tilde{a}(x)=a(x+x_0)$. Then $\tilde{a}$ satisfies:
(i) $\supp (\tilde{a})\subset B(0, r)$;
(ii)$\|\tilde{a}\|_{\infty}\leq |B(0,r)|^{-1}$ and
(iii)$\int\tilde{a}(x)dx=0$.

 It is enough to prove that $\int|\nH \tilde{a}(x)|dx<C$, where $C$ is independent of $\tilde{a}$.
Suppose that $\supp \tilde{a}\subset B(0,r)$ for $r>0$.
{\allowdisplaybreaks
\begin{eqnarray*}
&&\int_{\rz}\upsilon_n|\nH(\tilde{a})(x)|dx\\
&=&\int_{B(0,2r)}\l|\f{1}{|x|^n}\int_{|y|<|x|}\tilde{a}(y)dy\r|dx
+\int_{\rz\backslash B(0,2r)}\l|\f{1}{|x|^n}\int_{|y|<|x|}\tilde{a}(y)dy\r| dx\\
&=&\int_{B(0,2r)}\l|\f{1}{|x|^n}\int_{\{|y|<|x|\}\cap B(0,r)}\tilde{a}(y)dy\r| dx\\
&&\q\q\q+ \int_{\rz\backslash B(0,2r)}\l|\f{1}{|x|^n}\int_{\{|y|<|x|\}\cap B(0,r)}\tilde{a}(y)dy\r| dx\\
&:=&I_1+I_2.
\end{eqnarray*}
}In the last two estimates we used that $\supp \tilde{a}\subset B(0,r)$.
For $I_1$, we have the following estimate
\begin{eqnarray*}
I_1&\leq& \int_{B(0,2r)}\f{1}{|x|^n}\int_{\{|y|<|x|\}\cap B(0,r)}\l|\tilde{a}(y)\r| dydx\\
&\leq& \int_{B(0,2r)}\f{1}{|x|^n}\f{1}{|B(0,r)|}\int_{|y|<|x|}dy dx\\
&=&\f{1}{|B(0,r)|}\int_{B(0,2r)}\f{\upsilon_n|x|^n}{|x|^n}dx\\
&=&2^n\upsilon_n.
\end{eqnarray*}
For $I_2$, since $x\in \rz\backslash B(0,2r)$, we have $\{|y|<|x|\}\cap B(0,r)=B(0,r)$. Then
\begin{eqnarray*}
I_2&=&\int_{\rz\backslash B(0,2r)}\l|\f{1}{|x|^n}\int_{\{|y|<|x|\}\cap B(0,r)}\tilde{a}(y)dy\r|dx\\
&=&\int_{\rz\backslash B(0,2r)}\f{1}{|x|^n}\l|\int_{B(0,r)}\tilde{a}(y)dy\r|dx\\
&=& 0.
\end{eqnarray*}\end{proof}

\begin{remark} If we divide $\rz$ into two parts: $B(c_0r):=B(0,c_0 r)$ and $\rz\backslash B(c_0r)$. Here $c_0>0$.
Repeated the previous argument leads to for $c_0\geq 1$
$$I_1\leq c_0^n \upsilon_n,\ I_2=0.$$
But for $0<c_0<1$, we have
$$I_1\leq {c_0}^n \upsilon_n,\ I_2=\infty.$$
Therefore, we can obtain that the supremum whose the Hardy operator maps from $H^1$ to $L^1$ is 1 when $c_0=1$.
\end{remark}

Let $\chi_k=\chi_{C_k}$, $C_k=B_k\backslash{B_{k-1}}$, and $B_k=\{x:|x|\leq 2^k\}$. Let $1<p<\infty$. The Herz space $K_p$
is expressed as the intersection of the corresponding weighted
$L^{p}$ spaces. In \cite{Ga}, \cite{ly2} the space
$K_p$ is endowed with the expression
$$\|f\|_{K_p}:=\sum_{k\in\mathbb{Z}}2^{kn/{p'}}\|f\chi_k\|_p<\infty.$$

\begin{definition}\label{d31} Let $1<p<\infty$. A function $a$ on $\rz$ is said to be a central
$(1, p)$-atom if (1) $\supp (a)\subset B(0, r)$;\,\,
(2) $\|a\|_{p}\leq |B(0,r)|^{1/p-1}$;\,\,
(3) $\int a(x)dx=0$.
\end{definition}

\begin{lemma}\cite{ly2} \label{lem2} Let $f\in L^1(\rz)$ and $1<p<\infty$. Then $f\in HK_p(\rz)$ if and only if
$f$ can be represented as
$$f(x)=\sum_j\lz_j a_j(x),$$
each $a_j$ is a central $(1,p)$-atom and
$ \sum_j|\lz_j|<\infty$. Moreover,
$$\|f\|_{HK_p}\sim \inf\left\{\sum_i|\lz_i|\right\},$$
where the infimum is taken over all decompositions of $f$ as above.
\end{lemma}
In fact, $HK_p$ is the localization of $H_1$ at the origin. It is easy to see that the relation between $HK_p$ and $K_p$ is similar to one between $H^1$ and $L^1$. Thus we have

\begin{theorem}\label{thm3}
Let $1<p<\infty$. Then $\mathcal {H}$ maps from $HK_p(\rz)$ to $K_p(\rz)$.
\end{theorem}

To prove this theorem, we need the following proposition from \cite{ly2}.
\begin{proposition} \label{prop1}
Let $r>0$. Then
$$\|f(r\cdot)\|_{K_p}\sim r^{-n}\|f\|_{K_p}.$$
\end{proposition}

\noindent{\it Proof of Theorem \ref{thm3}.}\q By Lemma \ref{lem2}, it is easy to see that the proof of Theorem \ref{thm3} is reduced to show that
for any central$(1,p)$-atom $a$, $$\|\nH(a)\|_{K_p}\leq C,$$
where $C$ is independent of $a$. Let $\supp a\subset B(0,r)$. If we write $\mathbbm{a}(x)=r^n a(rx)$, then
$\mathbbm{a}(x)$ is a central $(1,p)$-atom supporting on unit ball $B(0,1)$. Substituting $y= r z$,
$dy=r^n dz$, we have
\begin{eqnarray*}
\nH(a)(rx)&=&\f{1}{\upsilon_n|rx|^n}\int_{|y|<|rx|}a(y)dy\\
 &=& r^{-n}\f{1}{\upsilon_n|x|^n}\int_{|z|<|x|}a(z r)r^ndz\\
 &=&r^{-n}\nH(\mathbbm{a})(x).
\end{eqnarray*}
By Proposition \ref{prop1}, we obtain
$$\|\nH(a)\|_{K_p}\sim\|\nH(\mathbbm{a})\|_{K_p}.$$

Therefore, it suffices to show
$$\|\nH(\mathbbm{a})\|_{K_p}\leq C,$$
where $C$ is independent of $\mathbbm{a}$.

By the definition of $K_p$, we write
\begin{eqnarray*}
\|\nH(\mathbbm{a})\|_{K_p}&=&\sum_{k\in\zz}2^{kn/{p'}}\|\nH(\mathbbm{a})\chi_k\|_p\\
&=&\sum_{k\leq 1}2^{kn/{p'}}\|\nH(\mathbbm{a})\chi_k\|_p+\sum_{k>1}2^{kn/{p'}}\|\nH(\mathbbm{a})\chi_k\|_p\\
&:=&J_1+J_2.
\end{eqnarray*}
$L^p(\rz)$-boundedness of $n$-dimensional Hardy operator leads to,
$$J_1\leq \f{p}{p-1}\sum_{k\leq 1}2^{kn/{p'}}\|\mathbbm{a}\|_{L^p(\rz)}\leq\f{p}{p-1}\sum_{k\leq 1}2^{kn/{p'}}=2^{n/{p'}}p'\f{2^{n/{p'}}}{2^{n/{p'}}-1}.$$

For $k>1$, and $|x|>2^{k-1}\geq 1$, it follows from $\{y:|y|<|x|\}\supseteq \{y: |y|<1\}$ that
\begin{eqnarray*}
\|\nH(\mathbbm{a})\chi_k\|_p^p=\int_{2^{k-1}<|x|\leq 2^k}\l|\f{1}{\upsilon_n|x|^n}\int_{|y|<1}\mathbbm{a}(y)dy\r|^p dx=0,
\end{eqnarray*}
where the last equality we used that $\int_{B(0,1)}\mathbbm{a}(y)dy=0$.\qed

\section{Endpoint estimates for commutators of n-dimensional Hardy operators}

In \cite{jn}, John and Nirenberg introduced the space $\BMO$. Define
$$\BMO(\rz)=\{f: \|f\|_{\BMO(\rz)}<\infty\},$$ where
$$\|f\|_{\BMO(\rz)}=\sup_{B}\frac{1}{|B|}\int_B |f(x)-f_B|dx.$$

If one regards two functions whose difference is a constant as one,
then the space $\BMO$ is a Banach space with respect to norm $\|\cdot\|_{\BMO(\rz)}$.
John-Nirenberg inequality
implies
$$\|b\|_{\BMO(\rz)}\sim\sup_{B\subset\mathbb{R}^{n}}\left\{\frac{1}{|B|}\int_{B}|b-b_{B}|^{p}\right\}^{1/p},$$
where $1<p<\infty$. An important fact is that the dual of $H^1$ is the space $\BMO$ (cf. \cite{S1}).

\begin{definition} \cite{FLL}, \cite{FLLW}\label{d41}  Let $b$ be a locally
integrable function on ${\mathbb{R}}^{n}$.
The commutators of $n$-dimensional  Hardy operators are defined by
$$[b,\mathcal {H}]=b{\mathcal{H}}f-{\mathcal{H}}(fb).$$
\end{definition}

In general, when symbols $b$ are in BMO$(\mathbb{R}^{n})$, the
properties of commutators are worse than those of the operators
themselves (for example the singularity, \cite{Pe1}). Therefore, we
imagine $[b,\mathcal {H}]$ is not bounded from $H^1(\rz)$ to $L^1(\rz)$. Furthermore, we conclude that the commutator maps form $H^1(\rz)$ to weak $L^1(\rz)$.

\begin{proposition} \label{prop2}
 If $b\in \BMO(\rz)$, then $[b,\mathcal {H}]$ is not bounded from $H^1(\rz)$ to $L^1(\rz)$.
\end{proposition}

\begin{proof}
We give the proof only for the case $n=1$. Take $b(x)=\chi_{(2,\infty)}(x)$ and $f_0(x)=\chi_{(0,2)}(x)-
\chi_{(-2,0)}(x)$. Then for $x>3$, we have the following estimate
\begin{eqnarray*}
\l|[b, H]f(x)\r|&=&\l|\frac{1}{|x|}\int_0^{x}(b(x)-b(y))f_0(y)dy\r|\\
&= &\frac{1}{x}\l|\int_0^2(1-0)\times 1 dy\r|=\frac{2}{x}.
\end{eqnarray*}
So we get
$$\int_{R}\l|[b,H]f(x)\r|dx \geq \int_{3}^{\infty}\frac{2}{x}dx=\infty.$$
\end{proof}
We state an interesting result for $[b, \nH]$ in what follows.
\begin{theorem} \label{thm5}
Let $b\in \BMO(\rz)$. Then $[b, \nH]$ maps $H^1(\rz)$ to $L^{1,\infty}(\rz)$.
 \end{theorem}
\begin{proof}
Similar to the proof of Theorem 2.1, it is enough to prove that $$\lz|\{x\in \rz:|[b, \nH](\tilde{a})(x)|>\lz\}|\leq C\|b\|_{\BMO}\|\tilde{a}\|_{H^1},$$
 where $C$ is independent of $\tilde{a}$.
Suppose that $\supp \tilde{a}\subset B(0,r)$ for $r>0$.
{\allowdisplaybreaks
\begin{eqnarray*}
[b, \nH](\tilde{a})(x)&=&\f{1}{\upsilon_n|x|^n}\int_{|y|<|x|}\tilde{a}(y)(b(x)-b(y))dy\\
&=&\f{1}{\upsilon_n|x|^n}\int_{\{|y|<|x|\}\cap B(0,r)}\tilde{a}(y)(b(x)-b_{B(0,r)})dy\\
&&\q+ \f{1}{\upsilon_n|x|^n}\int_{\{|y|<|x|\}\cap B(0,r)}\tilde{a}(y)(b_{B(0,r)}-b(y))dydx\\
&:=&L_1+L_2.
\end{eqnarray*}
}
In the last two estimates we used that $\supp \tilde{a}\subset B(0,r)$.

So we have
$$
|\{x\in \rz:|[b, \nH](\tilde{a})(x)|>\lz\}|\leq|\{x\in \rz: |L_1|>\lz/2\}|+|\{x\in \rz: |L_2|>\lz/2\}|.
$$
And
\begin{eqnarray*}
\frac{\lz}{2}|\{x\in \rz: |L_1|>\lz/2\}|&\leq & \int_{\rz}|L_1|dx\\
 &=&\int_{B(0,r)}|b(x)-b_{B(0,r)}||\nH (\tilde{a})(x)|dx\\
 &&+\int_{\rz\backslash B(0,r)}|b(x)-b_{B(0,r)}||\nH (\tilde{a})(x)|dx\\
&:=& L_{11}+L_{12}.
\end{eqnarray*}

Since $\nH$ is bounded on $L^p(\rz)$ for all $1<p<\infty$, together with H\"{o}lder's inequality, we obtain
\begin{eqnarray*}
L_{11}&\leq & \l(\int_{B(0,r)}|b(x)-b_{B(0,r)}|^2 dx\r)^{1/2}\l(\int_{B(0,r)}|\nH (\tilde{a})(x)|^2 dx\r)^{1/2}\\
&\leq& 2\l(\int_{B(0,r)}|b(x)-b_{B(0,r)}|^2 dx\r)^{1/2}\l(\int_{B(0,r)}|\tilde{a}(x)|^2 dx\r)^{1/2}\\
&=&2\frac{|B(0,r)|^{1/2}}{|B(0,r)|}\l(\int_{B(0,r)}|b(x)-b_{B(0,r)}|^2 dx\r)^{1/2}\\
&\leq &C\|b\|_{\BMO},
\end{eqnarray*}
where we used the condition (ii) of $\tilde{a}$.

If $x\in \rz\backslash B(0,r)$, then $\{y: |y|<|x|\}\cap\{y: |y|<r\}=\{y: |y|<r\}$. Using $\int\tilde{a}(y)dy=0$ and $\supp \tilde{a}=B(0,r)$,
we arrive at
$$\nH (\tilde{a})(x)=\f{1}{\upsilon_n|x|^n}\int_{|y|<r}\tilde{a}(y)dy=0,$$   and hence $L_{12}=0$.

We next estimate the  term $|\{x\in \rz: |L_2|>\lz/2\}|$ which is divided into two parts:
$|\{x\in B(0,r): |L_2|>\lz/2\}|$ and
$|\{x\in \rz\backslash B(0,r): |L_2|>\lz/2\}|$.
And
\begin{eqnarray*}
&&\f{2}{\lz}|\{x\in B(0,r): |L_2|>\lz/2\}|\leq \int_{B(0,r)}|L_2|dx \\
&=&\int_{B(0,r)}\l| \f{1}{\upsilon_n|x|^n}\int_{\{|y|<|x|\}\cap B(0,r)}\tilde{a}(y)(b_{B(0,r)}-b(y))dy\r| dx\\
&:=& L_{21}.
\end{eqnarray*}

Using H\"{o}lder's inequality, we get
\begin{eqnarray*}
L_{21}&\leq &\int_{B(0,r)} \f{1}{\upsilon_n|x|^n}\l(\int_{B(0,r)}|(b_{B(0,r)}-b(y))|^2dy\r)^{1/2} \l(\int_{|y|<|x|}|\tilde{a}(y)|^2dy\r)^{1/2}dx\\
&\leq & \int_{B(0,r)} \f{1}{\upsilon_n|x|^n}\f{|B(|x|)|^{1/2}}{|B(0,r)|^{1/2}}\l(\f{1}{|B(0,r)|}\int_{B(0,r)}|(b_{B(0,r)}-b(y))|^2dy\r)^{1/2}dx\\
&\leq &C\|b\|_{\BMO}.
\end{eqnarray*}

For the last term, we have

\begin{eqnarray*}
&&|\{x\in \rz\backslash B(0,r): |L_2|>\lz/2\}|\\
& = & \l|\l\{x\in \rz:  r<|x|\leq \l(\f{2}{\upsilon_n\lz}\int_{|y|<r}\l|b(y)-b_{B(0,r)}\r|\l|\tilde{a}(y)\r|dy\r)^{1/n}\r\}\r|\\
&=&\omega_n\int_{r}^{\l(\f{2}{\upsilon_n\lz}\int_{|y|<r}\l|b(y)-b_{B(0,r)}\r|\l|\tilde{a}(y)\r|dy\r)^{1/n}} t^{n-1} dt\\
&\leq& \frac{2}{\lz}\int_{|y|<r}\l|b(y)-b_{B(0,r)}\r|\l|\tilde{a}(y)\r|dy.
\end{eqnarray*}
By H\"{o}lder's inequality and the size estimate for $\tilde{a}$, the above integral is bounded by
\begin{eqnarray*}
&&\l(\int_{|y|<r}|(b_{B(0,r)}-b(y))|^2dy\r)^{1/2} \l(\int_{|y|<r}|\tilde{a}(y)|^2dy\r)^{1/2}dx\\
&\leq & C\|b\|_{\BMO}.
\end{eqnarray*}
Summing the estimates, we complete the proof.\end{proof}

Some comments are further put forward in the following part. In \cite{LY1}, Lu and Yang introduced the definition of central
$\BMO$ space as follows.

\begin{definition}\label{d42}
Let $1< q<\infty$. A
function $f\in L^q_{loc}(\rz)$ is said to belong to the space
$\CBMO^{q}(\rz)$(central {$\BMO$} space), if
$$\sup_{r>0}\left(\frac{1}{|B(0,r)|}\int_{B(0,r)}|f(x)-f_{B}|^{q} dx\right)^{1/q}=
\|b\|_{\CBMO^{q}(\rz} <\infty,$$ where
$$f_{B}=\frac{1}{|B(0,r)|}\int_{B(0,r)}f(x)dx.$$
\end{definition}
From the example of p.326 in \cite{Ko}, it is easy to see that
$\BMO({\mathbb{R}}^n)$ is a proper subset of ${\CBMO}^{q}(\rz)$,
where $1\leq q<\infty$. The space $\CBMO^{q}({\mathbb{R}}^n)$ can be regarded as a local
version of $\BMO(\rz)$ at the origin. The dual of $HK_p$ is $\CBMO^{p'}$, $1<p<\infty$ (cf. \cite{Ga}).

The same proof of Theorem 3.3 remains valid for $b\in \CBMO^q(\rz)$, $1<q<\infty$.
\begin{proposition} \label{cor1}
 Let $1<q<\infty$ and  $b\in \CBMO^q(\rz)$.  Then $[b, \nH]$ maps $H^1(\rz)$ to $L^{1,\infty}(\rz)$.
\end{proposition}

\noindent{\bf Acknowledgement}\quad This work is supported by the Key Laboratory of Mathematics and Complex System (Beijing Normal
University), Ministry of Education, China. The work is also supported by PCSIRT in University.

\bibliographystyle{amsplain}

\end{document}